\documentclass[12pt]{amsart}
\usepackage{amsmath,amsthm,amssymb,amsfonts}
\usepackage{theoremref}
\usepackage{hyperref}
\hypersetup{colorlinks=true}
\numberwithin{equation}{section}
\makeindex

\newcommand\xappa\kappa
\newcommand\yota\iota
\newcounter{consta}

\newcounter{constb}

\newcounter{constc}[section]

\newcommand{\vol}{\operatorname{vol}}

\newcommand{\R}{\mathbb{R}}

\newcommand{\Z}{\mathbb{Z}}

\DeclareFontFamily{OT1}{rsfs}{}
\DeclareFontShape{OT1}{rsfs}{n}{it}{<-> rsfs10}{}
\DeclareMathAlphabet{\mathscr}{OT1}{rsfs}{n}{it}
\swapnumbers
\newtheorem{thm}{Theorem}[section]
\newtheorem{lem}[thm]{Lemma}

\newtheorem{prop}[thm]{Proposition}
\newtheorem*{lem*}{Lemma}
\newtheorem*{thm*}{Theorem}
\newtheorem*{conj*}{Conjecture}
\newtheorem*{prop*}{Proposition}
\newtheorem{defn*}{Definition}

\newtheorem{ex}[thm]{Example}
\newtheorem{cor}[thm]{Corollary}

\theoremstyle{definition}
\newtheorem{defn}[thm]{Definition}
\theoremstyle{remark}
\newtheorem{rem}[thm]{Remark}

\newtheorem*{obs*}{Observation}
\newtheorem*{rem*}{Remark}
\theoremstyle{definition}\newtheorem*{acknowledgments}{Acknowledgments}

\begin{document}

\title[Margulis' inequality and equidistribution]{Margulis' Inequality for translates of horospherical orbits and applications to equidistribution}
\begin{abstract}
In this paper we develope a quantitative non-divergence theorem for translates of horospherical orbits, using the  technique of Margulis' inequality as developed by Eskin-Margulis-Mozes and Eskin-Margulis. 
As we use the Margulis' inequality, our results do not depend on the spectral gap of the action.
As an application of our techniques, we show that given a horospherical flow over the space of lattices, the horospherical orbit of every lattice defined over a number field, not contained in a proper rational parabolic subgroup is equidistributed with an effective rate.
\end{abstract}

\author{Asaf Katz}
\address{Department of Mathematics, University of Michigan, Ann Arbor, Michigan 48109, USA}
\email{asaf.katz@gmail.com}

\maketitle

\section{Introduction}

Let $G$ be  non-compact semi-simple Lie group, equipped with a probability distribution $\mu$.
Let $X$ be a $G$-space.
We may define the averaging operator $A_\mu$ over $L^{\infty}(X)$ as
\begin{equation*}
    A_{\mu}f(x)=\int_{G}f(g.x)d\mu(g).
\end{equation*}
\begin{defn}
We will say that the action $G\curvearrowright X$ satisfy a \emph{Margulis' inequality} if there exists some \emph{proper} function $f:X\to \mathbb{R}$ and constants $0\leq \alpha< 1$ and $\beta\in\mathbb{R}$ so that
\begin{equation*}
    A_{\mu}f(x) \leq \alpha\cdot f(x)+\beta,
\end{equation*}
for all $x\in X$.
\end{defn}
Margulis' inequalities were introduced by A. Eskin, G. Margulis and S. Mozes in~\cite{eskin-margulis-mozes} and generalized in~\cite{eskin-margulis,benoist-quint} in order to prove various non-divergence theorems.
The purpose of the paper is to introduce a new kind of Margulis' inequality, which is related to the one in~\cite{eskin-margulis-mozes} in order to achieve refined quantitative equidistribution statement for horospherical averages.

\begin{defn}
A subgroup $H\leq G$ is called horospherical subgroup with respect
to a diagonalizable element $a\in G$ if 
\[
H=\left\{ g\in G\mid a^{n}ga^{-n}\to e,\ n\to-\infty\right\} .
\]

It is easy to check that horospherical subgroups are nilpotent, connected,
simply-connected and consist of Ad-unipotent elements.

The main example for such subgroups is the upper unipotent matrices
in $SL_{2}\left(\mathbb{R}\right)$ with $a$ being a completely regular
diagonal matrix $a\in SL_{2}\left(\mathbb{R}\right)$ with diagonal entries satisfying $a_{1,1}>a_{2,2}$.
\end{defn}

Given a $G$-space $X$, define the convolution operator associated
to the $a_{\log R}$-translate of $H$ as:

\begin{equation*}
A_{R}f\left(x\right)=\int_{h\in B_{1}^{H}}f\left(a_{\log R}.h.x\right)d\mu_{H}\left(h\right),
\end{equation*}
 for every $f\in C_{b}\left(X\right).$
\begin{defn}

We say that the convolution operator $A_{R}$ has the uniform non-divergence
property if for all $\varepsilon>0$ there exists some compact subset
$K=K\left(\varepsilon\right)$ of $G/\Gamma$ such that for every
$x\in G/\Gamma$ there exists $M=M\left(x,\varepsilon\right)$ such
that 
\begin{equation}
\left(A_{R}^{\left(M\right)}\star\delta_{x}\right)\left(K\right)>1-\varepsilon.\label{eq:non-divergence}
\end{equation}
\end{defn}

A famous theorem of Dani-Margulis~\cite[Theorem~$1$]{dani-margulis-uni} shows that the $H$-average $$B_{R}f\left(x\right)=\frac{1}{\vol_{H}\left(B_{R}^{H}\right)}\int_{h\in B_{R}^{H}}f\left(h.x\right)d\mu_{H}\left(h\right)$$
enjoys the uniform non-divergence property except for points $x\in G/\Gamma$
such that $\overline{H.x}\neq X$ (and actually, a more precise statement
of the above, see for example~\cite[Theorems $1,2$]{dani-margulis-uni}).

The proof uses the polynomial behavior of the horospherical flow.

A quantitative statement was proven by Kleinbock-Margulis~\cite[Theorem~$5.2,5.3$]{kleinbock-margulis}. The main
theorem of Kleinbock-Margulis shows individual non-divergence, in
the sense that the compact subset $K$ given \emph{does depend} on
the origin point $x$ as well\footnote{this amounts to the $\rho$ estimate
in their Theorem.}. Such an alternative is inevitable, as the point
may lie in a ''small period'' up in the cusp and one may think of their parameter
$\rho$ as being some ''anchoring scale'' which Kleinbock-Margulis
choose to be related to the height of the origin point $x$ in $G/\Gamma$. 

The problem arises when $x$ is ''very Diophantine'' namely $H.x=X$
but ''locally in $H$'', $B_{R}^{H}.x$ is high in the cusp (for example,
being well-approximated by a periodic orbit of small orbit).

Nevertheless, for the question of equdistribution of horospherical
orbits, one is interested not only in the $H$-action, but in the
translates of $H$-orbits as represented in the $A_{R}$-averaging
operators, which give more tools to handle the non-divergence issues,
as the periodic orbits themselves are getting equidistributed under
the $a$-flow.

The question of considering translates of averages was risen in the
paper of Eskin-Margulis-Mozes~\cite{eskin-margulis-mozes} regarding certain quantitative estimates
associated with the Oppenheim conjecture, where they devised a technique
which is now termed ''Margulis' inequality'' in order to prove the
uniform non-divergence property, and generalizations of this technique
to the case of random walks generated by semisimple Lie subgroups
over homogeneous spaces have been given by Eskin-Margulis~\cite{eskin-margulis} and Benoist-Quint~\cite{benoist-quint}.

The main theorem of the paper proves a Margulis' inequality
for the $A_{R}$-operator.
For simplicity we just state here the results for the case of $G=SL_{n}(\mathbb{R}), \Gamma = SL_{n}(\mathbb{Z})$ and $X=G/\Gamma$. The general case is treated later in the text.
\begin{thm}
Let $H\leq G$ be a horospherical subgroup of $G$, with respect to a one-parameter group $A=\langle a\rangle\leq G$.
Then for all $R>1$ the averaging operator defined by $a_{\log R}.m_{B_{1}^{H}}$ satisfies a Margulis' inequality. Namely there exists a proper function $f:X\to \mathbb{R}$ such that for any $R\gg 1$, there exists $\alpha=\alpha(r)<1$ and $\beta=\beta(r)\in \mathbb{R}$ satisfying
\begin{equation}
    \int_{h\in B_{1}^{H}}f(a_{\log{R}}.h.x)dh \leq \alpha\cdot f(x)+\beta,
\end{equation}
for all $x\in X$.
\end{thm}

Furthermore, we show some applications to homogeneous dynamics, improving
on former works of the author, regarding horospherical equdistribution and rapid recurrence of semisimple periods.
In what comes next, we treat unimodular lattices in $\R^n$ as point in the homogenous space $X=SL_{n}(\R)/SL_n(\Z)$.
\begin{defn}
A lattice $x\in SL_{n}(\R)/SL_{n}(\Z)$ is said to be \emph{diophantine} if
\begin{equation*}
    \limsup_{R\to -\infty} \frac{\log\left(\alpha_{i}(a_{-\log R}.x.\mathbb{Z}^n \right)}{\log\left(\alpha_{i}(a_{-\log R}.\mathbb{Z}^n \right)} = \omega_i <1,
\end{equation*}
for all $1 \leq i \leq n$, where $\alpha_i(y\cdot\mathbb{Z}^n)$ is the $i$'th Euclidean minima of the lattice $y\cdot\mathbb{Z}^{n}$.
\end{defn}
This definition will be given in an extended form in Definition~\ref{def:dioph}.

\begin{thm}
Assume that $x\in X$ is \emph{diophatine}. Then there exists $\gamma=\gamma(\Gamma,\omega_1,\ldots,\omega_n)>0$ and $R_0=R_0(x)>0$ such that for all $R>R_0$ and any $f:X\to \mathbb{R}$ a bounded smooth function with vanishing integral
\begin{equation}
    \left\lvert \frac{1}{\vol_{H}(B_{R}^{H})} \int_{h\in B_{R}^{H}}f(h.x)dh \right\rvert \ll_{f} R^{-\gamma}.
\end{equation}
\end{thm}

\begin{acknowledgments}The author is indebted to Alex Eskin for
suggesting the application of Margulis' inequality to the problem of horospherical
equdistribution and his invaluable insights and comments regarding
this technique. The author would also like to thank Elon Lindenstrauss, Peter Sarnak, Dmitry Kleinbock and Amir Mohammadi for valuable conversations regarding Diophatine conditions for flows over homogeneous spaces. The author also like to thank Anthony Sanchez for careful reading of the paper and providing comments that vastly improved the presentation of the paper.
\end{acknowledgments}

\section{Preliminaries}
Let $G$ be a semisimple Lie group, $\Gamma\leq G$ a \emph{non-uniform lattice} and $X=G/\Gamma$ the resulting homogeneous space equipped with the probability measure $\mu$ which is induced by the Haar measure of $G$, properly normalized.
\begin{defn}
A function $u:X \to \mathbb{R}$ is said to be a \emph{Margulis function} associated to an averaging operator $A_{R}$ if
\begin{enumerate}
    \item u is a \emph{proper function}, namely $u(x)\to \infty$ as $x\to \infty$ in $X$.
    \item There exists constants $a<1, b>0$ such that 
    \begin{equation*}
        A_{R}.u(x) \leq a\cdot u(x)+b.
    \end{equation*}
\end{enumerate}
\end{defn}

For every $x\in X$, we may relate a probability measure on $X$ bh repeatedly averaging with respect to $A_{R}$, namely we define
\begin{equation*}
    \mu_{n,x}=A_{R}^{(n)}.\delta_{x},
\end{equation*}
where $\delta_{x}$ stands for the Dirac mass measure at $x$.

Assume that a Margulis' function exists for $A_{R}$, then the following Lemma holds
\begin{lem}
For \emph{every} $x\in X$, the sequence of probability measures $\left\{\mu_{n,x} \right\}\subset Prob(X)$ is tight, namely for every $\varepsilon$ there exists a compact subset $K_{\varepsilon}$ such that for every $n$,
\begin{equation*}
    \mu_{n,x}(K_{\varepsilon})>1-\varepsilon.
\end{equation*}
\end{lem}
The proof was given in~\cite{eskin-margulis} and we reproduce it here for the sake of completeness.
\begin{proof}
Repeatedly applying the averaging operator, combined with the Margulis' function's properties yields
\begin{equation*}
\begin{split}
    A_{R}^{(n)}.u(x) &\leq a\cdot A_{R}^{(n-1)}u(x)+b \\
    &\leq a^{n}u(x)+\sum_{i=0}^{n-1}a^{i}\cdot b \\
    &\leq a^{n}u(x) +b\sum_{i=0}^{\infty}a^{i} \\
    &= a^{n}u(x)+\frac{b}{1-a}.
\end{split}
\end{equation*}
For $n$ large enough (depending on $x$, we have that $A_{R}^{n}.u(x)\leq \frac{2b}{1-a}$.
As $u$ is proper function, it means that $K_{\varepsilon}={x\in X \mid u(x)\leq \frac{2b}{(1-a)\varepsilon}}\subset X$ is relatively compact.
Now we have by Markov's inequality
\begin{equation*}
    \mu_{n,x}(K_{\varepsilon}^{c}) \leq
    \frac{A_{R}^{(n)}u(x)}{\left(\frac{2b}{(1-a)\varepsilon}\right)} < \varepsilon.
\end{equation*}
\end{proof}

According to the construction given by Eskin-Margulis~\cite{eskin-margulis} of a Margulis function $u$ associated to an average $A_{R}$, in order to construct a Margulis function $u$, it is enough
to prove the following assertion:
\begin{lem}[Main Lemma]\label{lem:main-lem} Given a finite dimensional linear representation $\rho:G\to V$,
for any $0<\delta=\delta\left(R\right)\ll1$, there exists $C\left(\delta\right)<1$
such that for any $v\neq0$,\\
\begin{equation}
A_{R}\left( N_{\delta}\left(v\right)\right)^{-1}\leq C\left(\delta\right)\cdot\left(N_{\delta}\left(v\right)\right)^{-1},\label{eq:expansion-ineq}
\end{equation}
where $N_{\delta}\left(v\right)=\lVert v\rVert^{\delta}$ and $\lVert\cdot\rVert$
is a fixed norm on $V$.
\end{lem}

Once this Lemma is proven, a procedure discussed in~\cite[\S3]{eskin-margulis} yields a suitable Margulis function by optimizing various height functions on the group.

\section{Proof of Lemma~\ref{lem:main-lem}}
In this section we are going to prove Lemma~\ref{lem:main-lem} in our case.

Due to homogeneity of the norm, it is enough to prove the above inequality
for $v\in V$ with $\lVert v\rVert=1$.\\
Given an element $a\in G$, we may choose a maximal Cartan subgroup
$T\leq G$ defined over $\mathbb{Q}$ containing $a$, and consequently, we denote by $\Phi$
the related root system for $G$.

We may decompose the representation space $V$ as follows:
\begin{equation*}
V=\oplus_{\alpha\in\Phi}V_{\alpha},
\end{equation*}

where $V_{\alpha}=\left\{ v\in V\mid\rho\left(t\right).v=\alpha\left(t\right).v\right\} $
for some weight $\alpha$ (c.f.~\cite[$\S5.7$]{benoist-quint-book}).

We also define the following decomposition of $V$ into sub-spaces $$V=V_{-}\oplus V_{0}\oplus V_{+},$$
where 
$$V_{-}=\left\{ v\in V\mid\lVert\rho\left(a\right).v\rVert<\lVert v\rVert\right\}, $$ $$V_{0}=\left\{ v\in V\mid\lVert\rho\left(a\right).v\rVert=\lVert v\rVert\right\} $$
and $$V_{+}=~\left\{ v\in V\mid\lVert\rho\left(a\right).v\rVert>\lVert v\rVert\right\}. $$

Choosing a basis $\left\{ V_{\alpha,i}\right\} _{i=1}^{\dim V_{\alpha}}$
for each $V_{\alpha}$, we may write every $v\in V$ according to the decomposition as 
\begin{equation}
v=\sum_{\alpha\in\Phi}\sum_{i=1}^{\dim V_{\alpha}}c_{\alpha,i}\cdot v_{\alpha,i},\label{eq:vec-decomposition}
\end{equation}
 for some scalars $\left\{ c_{\alpha,i}\right\} _{\alpha\in\Phi,\ i\in\left\{ 1,\ldots,\dim V_{\alpha}\right\} }\subset\mathbb{C}$.

The following technical Lemma derives an explicit formula for the
$H$-action in this $G$-representation, emphasizing the polynomial
nature of the action.
\begin{lem}
$\rho\left(h\right).v=\sum_{\beta\in\Phi}\sum_{j=1}^{\dim V_{\beta}}f_{\beta,j}\left(\overline{t}\right)\cdot v_{\beta,j}$
for $h\in H$, for some fixed polynomials $f_{\beta,j}:\left[0,1\right]^{\dim H}\to\mathbb{R}$,
where $h=\exp\left(\sum_{i=1}^{\dim H}t_{i}\underline{h}_{i}\right)$
with $\left\{ \underline{h}_{i}\right\} _{i=1}^{\dim H}$ being a
fixed basis of $Lie\left(H\right)$.
\end{lem}

\begin{proof}
Write $h=\exp\left(\underline{h}\right)$ for $\underline{h}\in Lie\left(H\right)$,
where $\exp$ is a polynomial expression as $H$ is a nilpotent group,
$\exp\left(\underline{h}\right)=\sum_{j=0}^{L}b_{j}\cdot\underline{h}^{j}$
for some scalars $b_{j}\in\mathbb{C}$, and $L=L\left(G,H\right)\in\mathbb{N}$ with
$L\leq\dim H$.

Hence we may write 
\begin{align}
\rho\left(h\right).v & =\rho\left(\exp\left(\underline{h}\right)\right).v\nonumber \\
 & =\rho\left(\sum_{\ell=0}^{L}b_{\ell}\cdot\underline{h}^{\ell}\right).v\nonumber \\
 & =\sum_{\ell=0}^{L}b_{\ell}\cdot\rho\left(\underline{h}^{\ell}\right).v,\label{eq:exp-translation}
\end{align}
for some $L=L\left(G,H\right)$.

Writing $v$ according to the basis $v=\sum_{\alpha\in\Phi}\sum_{i=1}^{\dim V_{\alpha}}c_{\alpha,i}\cdot v_{\alpha,i}$
we get
\begin{equation*}
\sum_{\ell=0}^{L}b_{\ell}\cdot\rho\left(\underline{h}^{\ell}\right).v=\sum_{\alpha\in\Phi}\sum_{i=1}^{\dim V_{\alpha}}c_{\alpha,i}\left(\sum_{\ell=0}^{L}b_{\ell}\cdot\rho\left(\underline{h}^{\ell}\right).v_{\alpha,i}\right).
\end{equation*}
Note that $\underline{h}$ may be represented as $\underline{h}=\sum_{k=1}^{\dim H}t_{k}\cdot\underline{h}_{k}$
for some basis $\left\{ \underline{h}_{k}\right\} _{k=1}^{\dim H}$
of $Lie\left(H\right)$.

Hence $\rho\left(\underline{h}^{\ell}\right)=\sum_{\overline{\gamma}\in\left\{ 1,\ldots,\dim H\right\} ^{\ell}}t_{\overline{\gamma}}\rho\left(\underline{h}_{\overline{\gamma}}\right)$,
where we use the multi-index notion for $\overline{\gamma}$ and $t_{\overline{\gamma}}$ is the associated monomial.
For a multi-index $\overline{\gamma}$ , write $\rho\left(\underline{h}_{\overline{\gamma}}\right).v_{\alpha,i}$
according to the basis $\left\{v_{\beta,j} \right\}_{\beta\in \Phi, 1\leq j \leq \dim V_{\beta}}$ as
\begin{equation}
\rho\left(\underline{h}_{\overline{\gamma}}\right).v_{\alpha,i}=\sum_{\beta\in\Phi}\sum_{j=1}^{\dim V_{\beta}}d_{\beta,j,\alpha,i}\left(\underline{h}_{\overline{\gamma}}\right)\cdot v_{\beta,j}.\label{eq:action-decomposition}
\end{equation}
Therefore we get the following expression for $\rho\left(h\right).v$:
\begin{align*}
\rho\left(h\right).v & =\sum_{\alpha\in\Phi}\sum_{i=1}^{\dim V_{\alpha}}c_{\alpha,i}\left(\sum_{\ell=0}^{L}b_{\ell}\cdot\sum_{\overline{\gamma}\in\left\{ 1,\ldots,\dim H\right\} ^{\ell}}t_{\overline{\gamma}}\rho\left(\underline{h}_{\overline{\gamma}}\right).v_{\alpha,i}\right)\\
 & =\sum_{\alpha\in\Phi}\sum_{i=1}^{\dim V_{\alpha}}c_{\alpha,i}\left(\sum_{\ell=0}^{L}b_{\ell}\cdot\sum_{\overline{\gamma}\in\left\{ 1,\ldots,\dim H\right\} ^{\ell}}t_{\overline{\gamma}}.\sum_{\beta\in\Phi}\sum_{j=1}^{\dim V_{\beta}}d_{\beta,j,\alpha,i}\left(\underline{h}_{\overline{\gamma}}\right)\cdot v_{\beta,j}\right)\\
 & =\sum_{\beta\in\Phi}\sum_{j=1}^{\dim V_{\beta}}\left(\sum_{\ell=0}^{L}b_{\ell}\cdot\sum_{\overline{\gamma}\in\left\{ 1,\ldots,\dim H\right\} ^{\ell}}t_{\overline{\gamma}}\sum_{\alpha\in\Phi}\sum_{i=1}^{\dim V_{\alpha}}c_{\alpha,i}\cdot d_{\beta,j,\alpha,i}\left(\underline{h}_{\overline{\gamma}}\right)\right)\cdot v_{\beta,j.}
\end{align*}
We define the following polynomials $f_{\beta,j}\left(t_{1},\ldots,t_{\dim H}\right)$:
\begin{equation}
f_{\beta,j}\left(\overline{t}\right)=\sum_{\ell=0}^{L}b_{\ell}\cdot\sum_{\overline{\gamma}\in\left\{ 1,\ldots,\dim H\right\} ^{\ell}}t_{\overline{\gamma}}\sum_{\alpha\in\Phi}\sum_{i=1}^{\dim V_{\alpha}}c_{\alpha,i}\cdot d_{\beta,j,\alpha,i}\left(\underline{h}_{\overline{\gamma}}\right),\label{eq:coeff-poly}
\end{equation}
and we may write an explicit expression for $\rho\left(h\right).v$
as:
\begin{equation*}
\rho\left(h\right).v=\sum_{\beta\in\Phi}\sum_{j=1}^{\dim V_{\beta}}f_{\beta,j}\left(\overline{t}\right)\cdot v_{\beta,j}.
\end{equation*}

Note that $\deg\left(f_{\beta,j}\right)\leq L\leq\dim H$.
\end{proof}
We have the following explicit description for the $a$-action in this
representation, for a vector $v=\sum_{\beta\in\varphi}\sum_{j=1}^{\dim V_{\beta}}c_{\beta,j}v_{\beta,j}$:
\begin{equation*}
\rho\left(a\right).v=\sum_{\beta\in\Phi}e^{\chi_{a}\left(\beta\right)}\sum_{j=1}^{\dim V_{\beta}}c_{\beta,j}\cdot v_{\beta,j},
\end{equation*}
where $\chi_{a}(\beta)$ is the appropriate co-weight.

Therefore we get the explicit description of the full action of $a\in T$
and $h\in H$ as: 
\begin{equation}
\rho\left(a\cdot h\right).v=\sum_{\beta\in\Phi}\sum_{j=1}^{\dim V_{\beta}}e^{\chi_{a}\left(\beta\right)}\cdot f_{\beta,j}\left(\overline{t}\right)\cdot v_{\beta,j}.\label{eq:full-a-action}
\end{equation}

\begin{lem}[Anchor Lemma] For every $v\in V$, $v\neq0$, there exists some $\beta\in\Phi^{+}$
and $1\leq j\leq\dim V_{\beta}$ such that $\sup_{\overline{t}\in\left[0,1\right]^{\dim H}}\left|f_{\beta,j}\left(\overline{t}\right)\right|>0$.
\end{lem}

This lemma shows that as $H$ acts by raising the weights over $V$, there exists some \emph{positive weight} subspace $V_{\beta}$ for which the function $f_{\beta,j}(\overline{t})$ is not the constant function $0$.

\begin{proof}
Assume that $V$ is an irreducible $G$-representation. In the non-irreducible one, we move to the cyclic constituent generated by $\langle G.v\rangle \leq V$. It follows from the theorem of highest weight (c.f.~\cite[\S V, Theorem~$5.5$]{Knapp_1996}) that this sub-representation contains vector of positive weights.

For a small neighborhood $O$ of the identity $e\in G$ which generates
$G$ we may write $O$ as $O_{-}\cdot O_{0}\cdot O_{+}$ where $$O_{-}=\exp\left(\underline{O}_{-}\right), O_{0}=\exp\left(\underline{O}_{0}\right), O_{+}=\exp\left(\underline{O}_{+}\right),$$
for a proper partition $\underline{O}_{-}\oplus\underline{O}_{0}\oplus\underline{O}_{+}$
of a small neighborhood of the identity $\underline{0}\in Lie\left(G\right)$,
where 
$$\underline{O}_{-}\subset\oplus_{\alpha\in\Phi^{-}}Lie\left(G\right)_{\alpha},\underline{O}_{0}\subset\oplus_{\alpha\in\Phi^{0}}Lie\left(G\right)_{\alpha},\underline{O}_{+}\subset\oplus_{\alpha\in\Phi^{+}}Lie\left(G\right)_{\alpha}.$$

Note that by the definition of the weight spaces for the representation $(\rho,V)$, we have for $\underline{g}\in g_{\alpha}$ and $v\in V_{\beta}$
that 
\begin{equation*}
\rho\left(\underline{g}\right).v\in V_{\alpha+\beta},
\end{equation*}
see also \cite[\S V, Proposition~$5.4$]{Knapp_1996}.
Moreover, $H=\exp\left(Lie\left(H\right)\right)$, with $Lie\left(H\right)\leq\oplus_{\alpha\in\Phi^{+}}Lie\left(G\right)_{\alpha}$.

Assume the contrary, namely for every $\beta\in\Phi^{+}$ and every
$1\leq~j\leq~\dim V_{\beta}$ and every $\overline{t}\in\left[0,1\right]^{\dim H}$,
we have that $f_{\beta,j}\left(\overline{t}\right)\equiv0$.

This assumption forces $v$ to be supported only in $V_{-}\oplus~V_{0}$.

As $f_{\beta,j}\left(\overline{t}\right)$ are polynomial functions,
this means that the polynomials $f_{\beta,j}$ are identically zero for any $\overline{t}$, for all $\beta\in\Phi^{+}$ and $1\leq j \leq \dim V_{\beta}$.

Hence the projection to the positive weight spaces satisfies $$\Pi_{\Phi^{+}}\left(\rho\left(H\right).v\right)=0.$$


Moreover, an easy weight calculation demonstrates that $\rho\left(O_{-}\right).V_{-}\oplus V_{0},\rho\left(O_{0}\right).V_{-}\oplus V_{0}\subset V_{-}\oplus V_{0}$.

We conclude that $V_{-}\oplus V_{0}$ is a \emph{$G$-invariant subspace of $V$}, as $G$ is generated by $O_{-}O_{0}O_{+}$. Therefore
$V=~V_{-}\oplus~V_{0}$ as we assumed the representation is irreducible, which leads to a contradiction.
\end{proof}

Denote $M_{\rho}$ to be 
\begin{equation}
M_{\rho}=\min_{v\in V,\ \lVert v\rVert=1}\left(\max_{\beta\in\Phi^{+},1\leq j\leq \dim V_{\beta}}\max_{\overline{t}\in\left[0,1\right]^{\dim H}}\left|f_{\beta,j}\left(\overline{t}\right)\right|\right).\label{eq:anchor-estimate}
\end{equation}
By the previous Lemma, $M_{\rho}>0$.
Essentially this observation shows that the projection to at-least one \emph{positive} weight space $V_{\beta}$ is non-zero.
As this is a positive weight space, the norm of this projection grows under the $a$-action.

\section{Derivation of Margulis' inequality}

During the course of the proof of the Margulis' inequality, we will need a sub-level estimate for polynomial functions. Such a result follows from Remez' inequality and verifies the $(C,\alpha)$-good property of polynomial maps as defined in Kleinbock-Margulis~\cite{kleinbock-margulis} and Eskin-Mozes-Shah~\cite{Eskin_Mozes_Shah_1996}. We provide the details bellow.

\begin{thm}[Remez' inequality~\cite{brudnyui1973extremal} Theorem~$2$]
	Let $B\subset\mathbb{R}^{n}$ be a non-empty convex subset, $f:B\to\mathbb{R}$ be polynomial of degree $d$. For $\epsilon>0$ we set $Z_{B,\epsilon}=~\left\{x\in B \mid \lvert f(x)\rvert < \varepsilon\right\}$, then:
	\begin{equation*}
	\sup_{x\in B}\lvert f(x) \rvert \leq \varepsilon \cdot T_{d}\left(\frac{1+(1-m(Z_{B,\epsilon})/m(B))^{1/n}}{1-(1-m(Z_{B,\epsilon})/m(B))^{1/n}}\right),
	\end{equation*}
	where $T_{d}$ is the $d$-th Chebyshev polynomial of the first kind.
\end{thm}
As a corollary one may deduce the following estimate:
\begin{cor}[Sub-level estimate]
	Let $B\subset\mathbb{R}^{n_1}$ be a non-empty convex subset, $f:B\to\mathbb{R}^{n_2}$ be polynomial of degree $d$, one has:
	\begin{equation}\label{eq:remez-ineq}
	m\left\{x\in B \mid \lVert f \rVert<\varepsilon \right\} \leq 4\cdot n_1 \cdot \left(\frac{\varepsilon}{\sup_{x\in B}\left\{\lVert f(x)\rVert \right\}}\right)^{1/d}\cdot m(B),
	\end{equation}
	where $m$ denotes the Lebesgue measure over $\mathbb{R}^{n_1}$.
\end{cor}

\subsection{Main proof of the Margulis' inequality}
We may now begin our proof of the Margulis' inequality.
Considering $G=SL_{n}(\R)$, we exhibit a specific set of representations as follows: $(\rho_{1},\R^n)$ - the representation of $G$ on $\R^n$ by matrix multiplication. For all $2\leq i\leq n$ we extend this definition to a representation $(\rho_i, \bigwedge^{i}\R)$ by considering exterior products.

In order to prove a Margulis' inequality, the following key lemma due to Eskin-Margulis~\cite{eskin-margulis} will be used:
\begin{lem}[\cite{eskin-margulis}, Proposition $2.6$ and Section \S3]\label{lem:E-M-main-lem}
Consider the representations $\left\{\rho_{i}\right\}_{i=1}^{n}$. Let $(\rho,V)$ be any of those representations. Assume that for all  $\delta>0$ small enough there exists some $c=c(R,\delta)<1$ such that for all $0\neq v\in V$ the following holds:
\begin{equation*}
    \int_{h\in B_{1}^{H}}\lVert a_{\log R}.h.v\rVert^{-\delta}dh \leq c\cdot \lVert v\rVert.
\end{equation*}
Then $A_{\mu}$ satisfy a Margulis' inequality for constants $\alpha=\alpha(c)<1, \beta=~\beta(c)\in\mathbb{R}$.
\end{lem}

By writing the average explicitly using the formulas achieved above we get:

\begin{align*}
\int_{h\in B_{1}^{H}}\frac{1}{\lVert\rho\left(a.h\right).v\rVert^{\delta}}dh & =\int_{\overline{t}\in\left[0,1\right]^{\dim H}}\frac{1}{\left(\sum_{\beta\in\Phi}\sum_{j=1}^{\dim V_{\beta}}e^{2\chi_{a}\left(\beta\right)}\cdot f_{\beta,j}\left(\overline{t}\right)^{2}\right)^{\frac{\delta}{2}}}d\overline{t}.
\end{align*}

Fix some $\beta,j$ where $\beta\in\Phi^{+},1\leq j\leq\dim V_{\beta}$
is an index such that 
\begin{equation}\label{eq:coeff-estimate}
\sup_{\overline{t}\in\left[0,1\right]^{\dim H}}\left|f_{\beta,j}\left(\overline{t}\right)\right|\geq M_{\rho}.
\end{equation}

Let $\tau$ be a small positive number, to be specified later. In the proof that follows, all constant of the form $C_{\star}$ will be absolute in terms of $(\rho,V)$ and $H$, \emph{not pertaining} to any particular $v\in V$.

Let $A\subset\left\{ \overline{t}\in\left[0,1\right]^{\dim H}\right\} $
be the set of $\overline{t}$ such that $f_{\beta,j}\left(\overline{t}\right)^{2}<\tau^{2}$,
and denote $B$ to be its complement.
We note that as $\rho(h)$ is a continuous invertible operator for any $h\in H$, we have that there exists some $C=C(\rho)>0$ such that for all $v\in V$, $\left\lVert v\right\rVert=1$ we have
\begin{equation*}
    \min_{h\in \exp\left([0,1]^{\dim H}\right)} \left\lVert \rho(h).v\right\rVert = \min_{\underline{t}\in [0,1]^{\dim H}}\left\lVert \sum_{\beta\in\Phi}\sum_{j=1}^{\dim V_{\beta}}f_{\beta,j}(\underline{t})^{2} \right\rVert \geq C.
\end{equation*}

We may estimate the integral as follows 

\begin{align}\label{eq:margulis-ineq-estimate}
\begin{split}
&\int_{\overline{t}\in\left[0,1\right]^{\dim H}}\frac{1}{\left(\sum_{\beta\in\Phi}\sum_{j=1}^{\dim V_{\beta}}e^{2\chi_{a}\left(\beta\right)}\cdot f_{\beta,j}\left(\overline{t}\right)^{2}\right)^{\frac{\delta}{2}}}d\overline{t} \\
&\ =\int_{A}\frac{1}{\left(\sum_{\beta\in\Phi}\sum_{j=1}^{\dim V_{\beta}}e^{2\chi_{a}\left(\beta\right)}\cdot f_{\beta,j}\left(\overline{t}\right)^{2}\right)^{\frac{\delta}{2}}}d\overline{t}
+ \int_{B}\frac{1}{\left(\sum_{\beta\in\Phi}\sum_{j=1}^{\dim V_{\beta}}e^{2\chi_{a}\left(\beta\right)}\cdot f_{\beta,j}\left(\overline{t}\right)^{2}\right)^{\frac{\delta}{2}}}d\overline{t}\\
&\ \leq\int_{A}\frac{1}{\left(\sum_{\beta\in\Phi}\sum_{j=1}^{\dim V_{\beta}}e^{2\chi_{a}\left(\beta\right)}\cdot f_{\beta,j}\left(\overline{t}\right)^{2}\right)^{\frac{\delta}{2}}}d\overline{t} +\int_{B}\frac{1}{\left(e^{2\chi_{a}\left(\beta\right)}\cdot f_{\beta,j}\left(\overline{t}\right)^{2}\right)^{\frac{\delta}{2}}}d\overline{t}\\
&\ \leq\int_{A}e^{-\delta\cdot \chi_{a}\left(\min\Phi^{-}\right)}\cdot C^{-\delta}d\overline{t} +\int_{B}\frac{1}{\left(e^{2\chi_{a}\left(\beta\right)}\cdot f_{\beta,j}\left(\overline{t}\right)^{2}\right)^{\frac{\delta}{2}}}d\overline{t}\\
&\ \leq m\left(A\right)e^{-\delta\cdot\chi_{a}\left(\min\Phi^{-}\right)}\cdot C^{-\delta} +m\left(B\right)e^{-\delta\cdot\chi_{a}\left(\beta\right)}\cdot\tau^{-\delta}\\
&\ \leq m\left(A\right)e^{-\delta\cdot\chi_{a}\left(\min\Phi^{-}\right)}\cdot C^{-\delta}+e^{-\delta\cdot\chi_{a}\left(\beta\right)}\cdot\tau^{-\delta}.
\end{split}
\end{align}

By Remez' inequality~\eqref{eq:remez-ineq} - $m\left(A\right)<C_{1}\cdot\left(\frac{\tau}{M_{\rho}}\right)^{\alpha}$,
so we have the following bound - 

\begin{equation*}
\begin{split}
    &\int_{\overline{t}\in\left[0,1\right]^{\dim H}}\frac{1}{\left(\sum_{\beta\in\Phi}\sum_{j=1}^{\dim V_{\beta}}e^{2\chi_{a}\left(\beta\right)}\cdot f_{\beta,j}\left(\overline{t}\right)^{2}\right)^{\frac{\delta}{2}}}d\overline{t} \\
    &\ \ \leq C_{1}\cdot\left(\frac{\tau}{M_{\rho}}\right)^{\alpha}e^{-\delta\cdot\chi_{a}\left(\min\Phi^{-}\right)}\cdot C^{-\delta}+e^{-\delta\cdot\chi_{a}\left(\beta\right)}\cdot\tau^{-\delta}.
\end{split}
\end{equation*}

Optimizing for $\tau$ gives: 
\begin{equation*}
\tau^{\delta+\alpha}=C_{2}\cdot e^{\delta\left(\chi_{a}\left(\min\Phi^{-}-\beta\right)\right)},
\end{equation*}
or equivalently 
\begin{equation*}
\tau=C_{3}\cdot e^{\frac{\delta}{\delta+\alpha}\chi_{a}\left(\min\Phi^{-}-\beta\right)},
\end{equation*}
leading to a bound over the right hand side of~\eqref{eq:margulis-ineq-estimate} by
\begin{align*}
C_{a,\delta} &= C_{1}\cdot\left(\frac{\tau}{M_{\rho}}\right)^{\alpha}e^{-\delta\cdot\chi_{a}\left(\min\Phi^{-}\right)}\cdot C^{-\delta}+e^{-\delta\cdot\chi_{a}\left(\beta\right)}\cdot\tau^{-\delta} \\ & \leq 2C_{4}\cdot  e^{\frac{-\delta^{2}}{\delta+\alpha}\chi_{a}\left(\min\Phi^{-}-\beta\right)}\cdot e^{-\delta\cdot\chi_{a}(\beta)} \\
 & =C_{5}\cdot e^{-\delta\cdot\chi_{a}\left((1-\frac{\delta}{\delta+\alpha})\cdot\beta +\frac{\delta}{\delta+\alpha}\cdot\min\Phi^{-}\right)}.
\end{align*}
Hence for
\begin{equation*}
(1-\frac{\delta}{\delta+\alpha})\cdot\chi_{a}(\beta) < -\frac{\delta}{\delta+\alpha}\cdot\chi_{a}(\min\Phi^{-}),
\end{equation*}
or equivalently
\begin{equation}\label{eq:delta-cond}
    \delta < \alpha\cdot \frac{-\chi_{a}\left(\min \Phi^{-}\right)}{\chi_{a}(\beta)},
\end{equation}
we get a proper decay, as long as we let $\lVert a \rVert$ grow, and in particular, for any choice of admissible $\alpha$ per~\eqref{eq:delta-cond} we have that for $\lVert a \rVert \gg_{\delta} 0$, $C_{a,\delta}<1$.

This verifies Lemmma~\ref{lem:E-M-main-lem}.

\begin{ex}\label{ex:SL2-example}
For the diagonal matrix $a=\text{Diag}(e^{1/2},e^{-1/2})$, with the standard representation of $SL_{2}(\mathbb{R})$ on $\mathbb{R}^{2}$, using~\eqref{eq:delta-cond} we get $$\beta=1/2, \min\Phi^{-}=-1/2,$$ and the polynomial controlling $m(A)$ is linear, so $\alpha=1$. This leads to an estimate of $$\delta<1.$$
\end{ex}

We end this section with a brief discussion about optimization.
Remez' inequality gives sharp results for a general class of polynomials. 
In the following lemma, we show that we need to only consider polynomials of special kind to our purposes, improving upon Remez' inequality.

\begin{prop}
	Given a representation $\left(\rho,V\right)$, there exists a number $p=p(\rho)$ such that the polynomial $f_{\beta,j}(\overline{t})$ for $\beta,j$ being the subspace of maximal weight, is non-zero and there is a multi-index $\overline{\gamma}$ of length bounded by $p$ so that $\left\lvert\frac{\partial^{\lvert \gamma\rvert}}{\partial \overline{\gamma}}f_{\beta,j}(\overline{t})\right\rvert>0$.
\end{prop}
\begin{proof}
	Let $\min\Phi$ be a minimal weight vector in $V$.
	We define $p$ to be the length of the shortest weight string starting in $\min\Phi$ ending in $\max\Phi$.
	As $\min\Phi$ is the minimal weight, for any other weight $\eta\in\Phi$ we have that the shortest length of a weight string starting at $\eta$ and ending in $\max\Phi$ is bounded by $p$ as well.
	Given a vector $v$ for which the lowest supported weight is $\eta$, there exists a weight string starting at $\eta$ ending in $\max\Phi$ of length bounded by $p$.
	Let $\overline{\gamma}$ by the associated root word.
	We claim that $f_{\beta,j}(\overline{t})$ contains the monomial $t^{\overline{\gamma}}$, for $\beta=~\max\Phi$.
	If not, define the vector $v'$ to be just the projection of $v$ into the $V_\eta$ weight space.
	By assumption, $\rho(H).v'$ would not be supported in $V_{\beta}$, and by the same reasoning as the anchor Lemma, $\rho(G).v'$ would generate a sub-representation, contradicting irreducibility.
	Therefore, $\left\lvert\frac{\partial^{\lvert \gamma\rvert}}{\partial \overline{\gamma}}f_{\beta,j}(\overline{t})\right\rvert = C\neq 0$.
\end{proof}
 Hence we may modify the sub-level estimate using the classical gradient inequality (i.e.~\cite{carbery}), giving $\alpha=1/p$.

This modification leads to an improvement in the decay rate estimate, and moreover to the assumption that $\beta$ is the maximal weight in~\eqref{eq:delta-cond}, leading to the constraint
\begin{equation*}
\delta<\alpha \cdot \frac{\chi_{a}(\min \Phi^{-})}{\chi_{a}(\max \Phi^{+})}
\end{equation*}
with a decay rate bounded by $-\frac{1/p}{\delta+1/p}\chi_{a}\left(\max \Phi^{+}\right) - \frac{\delta^{2}}{\delta+1/p}\chi_{a}(\min \Phi^{-})$.

\section{Some applications}

\subsection{Recurrence of translates of unipotent trajectories.}

The main purpose of this subsection is to carry out the strategy given
in Eskin-Margulis in-order to derive a quantitative explicit uniform
non-divergence statement for the $A_{R}$-operator.

\begin{lem*}
[Eskin-Margulis~\cite{eskin-margulis} Lemma~$3.1$] - Suppose that there exists a positive continuous
function $u:G/\Gamma\to\mathbb{R}$ with the following properties:
\end{lem*}
\begin{enumerate}
\item (Properness) $u\left(x\right)\to\infty$ as $x\to\infty$.
\item (Margulis inequality) There exists $c_{1}<1$ and $b>0$ and $n>0$
such that for any $x\in G/\Gamma$
\begin{equation}
A_{R}^{\left(n\right)}u\left(x\right)\leq c_{1}\cdot u\left(x\right)+b,\label{eq:quant-margulis-ineq}
\end{equation}
where $A_{R}^{\left(n\right)}$ means the $n$-fold composition of $A_{R}$.
\end{enumerate}
Then $A_{R}$ has the quantitative recurrence property. The proof of this Lemma follows verbatim from Eskin-Margulis~\cite[Lemma~$3.2$]{eskin-margulis}.

The proof \cite[Lemma~$3.2$]{eskin-margulis} is effective, establishing
the following values for $K,M$:

\begin{equation}
K=\left\{ y\in G/\Gamma\mid u\left(y\right)<\frac{2\frac{b}{1-c_{1}}}{\varepsilon}\right\} ,\ M\left(x\right)=\frac{\log\left(\frac{u\left(x\right)}{\frac{b}{1-c_{1}}}\right)}{\log\left(\frac{1}{c_{1}}\right)}.\label{eq:K-M-formulas}
\end{equation}

Moreover, Eskin-Margulis have demonstrated a construction of a suitable
function (c.f. \cite[Section $3.1$]{eskin-margulis}) for $G/\Gamma=SL_{n}\left(\mathbb{R}\right)/SL_{n}\left(\mathbb{Z}\right)$
as follows:

For every $\Lambda\in G/\Gamma$, we define $u_{\epsilon,\delta}\left(\Lambda\right)=\sum_{i=0}^{n}\epsilon^{i\left(n-i\right)}\alpha_{i}\left(\Lambda\right)^{\delta}$,
where 
\begin{equation*}
\alpha_{i}\left(\Lambda\right)=\sup\left\{ \frac{1}{\text{covol}\left(L\right)}\mid L\ \text{is a \ensuremath{\Lambda}-rational subspace of dimension \ensuremath{i}}\right\} 
\end{equation*}
 for $0\leq i\leq d$, then for sufficiently small $\epsilon,\delta>0$,
$u_{\epsilon,\delta}$ satisfies the Eskin-Margulis Lemma.

We will now determine a suitable $\epsilon,\delta$ for our case and
as a result, get an estimate for $b,c_{1}$.

Let $\delta_{0}$ be a number so that the Lemma~\ref{lem:main-lem} is satisfied. 

Per the computations in Eskin-Margulis, picking $\delta=\frac{\delta_{0}}{n}$,
$\epsilon=\frac{1-C\left(\delta\right)}{3n\omega^{2}}$ we get $c_{1}=\frac{1+2C\left(\delta\right)}{3},\ b=2$,
hence 

\begin{equation}
\begin{split}
    &K=\left\{ y\in G/\Gamma\mid u\left(y\right)<\frac{6}{\varepsilon\left(1-C\left(\delta\right)\right)}\right\} ,\\ &M\left(x\right)=\frac{\log\left(\frac{u\left(x\right)}{\frac{1}{1-C\left(\delta\right)}}\right)}{\log\left(\frac{3}{1+2C\left(\delta\right)}\right)}=\frac{\log\left(u\left(x\right)\right)+\log\left(1-C\left(\delta\right)\right)}{\log\left(3\right)-\log\left(1+2C\left(\delta\right)\right)}.
\end{split}\label{eq:explicit-K-M}
\end{equation}

Fixing $C\left(\delta\right)<1$, we have the following recurrence
estimate:

For $m>\frac{\log\left(u\left(x\right)\right)+\log\left(1-C\left(\delta\right)\right)}{\log\left(3\right)-\log\left(1+2C\left(\delta\right)\right)}$:
\begin{equation}
\left(A_{R}^{\left(m\right)}\star\delta_{x}\right)\left\{ y\in G/\Gamma\mid\sum_{i=0}^{n}\left(\frac{1-C\left(\delta\right)}{3n\omega^{2}}\right)^{i\left(n-i\right)}\alpha_{i}\left(y\right)^{\delta}<\frac{6}{\varepsilon\left(1-C\left(\delta\right)\right)}\right\} >1-\varepsilon.\label{eq:explicit-u}
\end{equation}

In particular, we have that 
\begin{equation*}
\left(A_{R}^{\left(m\right)}\star\delta_{x}\right)\left\{ y\in G/\Gamma\mid\left(\frac{1-C\left(\delta\right)}{3n\omega^{2}}\right)^{n-1}\alpha_{1}\left(y\right)^{\delta}<\frac{6}{\varepsilon\left(1-C\left(\delta\right)\right)}\right\} >1-\varepsilon,
\end{equation*}
or by re-arranging 
\begin{equation*}
\left(A_{R}^{\left(m\right)}\star\delta_{x}\right)\left\{ y\in G/\Gamma\mid\alpha_{1}\left(y\right)<\left(\frac{C_{\dim}\left(n\right)}{\varepsilon\left(1-C\left(\delta\right)\right)^{n}}\right)^{1/\delta}\right\} >1-\varepsilon.
\end{equation*}

Using the relation between $\alpha_{1}\left(y\right)$ and the injectivity
radius at $y$ we may deduce that 
\begin{equation*}
\mu^{\left(m\right)}\left\{ y\in G/\Gamma\mid \text{InjRad}\left(y\right)>\left(\frac{\varepsilon\left(1-C\left(\delta\right)\right)^{n}}{C_{\dim}\left(n\right)}\right)^{n/\delta}\right\} >1-\varepsilon.
\end{equation*}

Relation to horospherical equdistribution:

As $B_{r}^{H}a_{\log t}=a_{\log t}B_{r/t}^{H}$ we have the following
formula:

\begin{equation*}
\left(aB_{r}^{H}\right)\cdot\left(aB_{r}^{H}\right)=a_{2\log t}B_{r/t}^{H}\cdot B_{r}^{H}\subset a_{2\log t}B_{r\left(1+1/t\right)}^{H},
\end{equation*}

and by induction we deduce 
\begin{align*}
\left(aB_{r}^{H}\right)\cdot\left(aB_{r}^{H}\right)\cdots\left(aB_{r}^{H}\right) & =a_{m\log t}B_{r/t^{m-1}}^{H}\cdot B_{r/t^{m-2}}^{H}\cdots\cdot B_{r}^{H}\\
 & \subset a_{m\log t}B_{r\sum_{i=0}^{m-1}t^{-i}}^{H}\\
 & =a_{m\log t}B_{r\left(\frac{1-t^{-m}}{1-t}\right)}^{H}\\
 & \subset a_{m\log t}B_{r\cdot\frac{t^{m}-1}{t^{m}\left(t-1\right)}}^{H}\\
 & \subset a_{m\log t}B_{r\cdot\frac{t}{\left(t-1\right)}}^{H}\\
 & \subset a_{m\log t}B_{2r}^{H}
\end{align*}

for $t\geq2$, where 
\begin{equation*}
    \begin{split}
        m&=\frac{\log\left(u\left(x\right)\right)+\log\left(1-C\left(\delta\right)\right)}{\log\left(3\right)-\log\left(1+2C\left(\delta\right)\right)}\\
        &\leq\frac{\log\left(u\left(x\right)\right)+\log\left(1-C\left(\delta\right)\right)}{\log\left(3\right)-\log\left(1+2C\left(\delta\right)\right)}\\
        &=\frac{\log\left(u\left(x\right)\right)+\log\left(1-C\left(\delta\right)\right)}{\log\left(3\right)-\log\left(1+2C\left(\delta\right)\right)}.
    \end{split}
\end{equation*}

We may bound $\log\left(u\left(x\right)\right)$ in a fairly crude manner
as 
\begin{align*}
\log\left(u\left(x\right)\right) & =\log\left(\sum_{i=0}^{n}\epsilon^{i\left(n-i\right)}\alpha_{i}\left(x\right)\right)\\
 & \leq\log\left(\sum_{i=0}^{n}\epsilon^{i\left(n-i\right)}\alpha\left(x\right)\right)\\
 & =\log\left(\alpha\left(x\right)\right)+\log\left(\sum_{i=0}^{n}\epsilon^{i\left(n-i\right)}\right).
\end{align*}

Therefore $m\leq\frac{\log\left(u\left(x\right)\right)+D\left(\delta\right)}{E\left(\delta\right)}$,
so $m\log t=\log\left(t^{m}\right)\approx\log\left(t^{\frac{\log\left(u\left(x\right)\right)+D\left(\delta\right)}{E\left(\delta\right)}}\right)$.

We have as an immideate corollary a different proof to the following Lemma, which played a key-role in the effective period equidistribution theorem of Einsidler-Margulis-Venkatesh~\cite[Appendix $B$]{emv}, which originally was proven by methods about recurrence of unipotent flows.
\begin{cor}
Let $S\leq G$ be a semisimple connected subgroup such that the centralizer of $Lie(S)$ in $Lie(G)$ is trivial.
Then there exists a compact subset $K\subset G$ such that \emph{every} closed $S$ orbit $S.x$ intersects $K$.
\end{cor}
\begin{proof}
Consider the operator $A_{R}$ defined over $S$, when we extend the appropriate Cartan subgroup to all of $G$.
Following the analysis from the previous section, the only option for the averages $A_{R}^{(n)}.f(x)$ to not decay is if $x$ if there exists some wedge representation, $x$ is in a part of the weight space such that $H.x$ does not have any positive weight component.
Notice we may do the same construction with a negative Cartan element and the opposite horospherical.
Hence we see that all the orbit $S.x$ is contained in the negative weight space.
But that means that $S$ is contained in some proper parabolic $P$.
Now consider the unipotent radical of $P$.
It must be strictly larger than $H$.
As a result, we see that each unipotent element $\mathfrak{u} \in Lie(Rad^{U}(P))$ commutes with all the unipotent elements in $Lie(S)$.
As $S$ is semisimple, it is generated by its unipotents, and so $Lie(S)$ has a non-trivial centralizer in $Lie(G)$, in contradiction.
Consider the parabolic associated to those negative weight spaces.
We have that $N_{G}(H)=P$.
But as $N_{G}(S)=S$ (due to the Lie algebra condition), we see that $P=S$.

If so, pick some unipotent element $\mathfrak{u}\in Lie(G)$ which amounts to the top weight space.
We claim that we must have that $[\mathfrak{u},\mathfrak{s}]=0$ for each unipotent element $\mathfrak{s}$ of $Lie(S)$.

\end{proof}

\subsection{Effective equidistribution of horospherical orbits under diophantine conditions.}

The general quantitative horospherical equidistribution theorem has been proved by the first named author in~\cite{katz}, see also~\cite{strombergsson, mcadam} for the $SL_{2}$-case and the abelian case, respectively.

The main difference between \cite{strombergsson} and the results in \cite{katz,mcadam} is that that the former one uses a ``diophantine condition'' defined by the divergence rate of $a_{-\log T}.x_0$ while the later use a diophantine condition related to the recurrence time of the point $a_{-\log T}.x_0$ under the horospherical flow, which has a strong relation to the non-divergence theorem, see \cite[Corollary~$3.3$]{lmms}.

Verifying the diophantine condition given in \cite{katz,mcadam} effectively is hard (indeed, for $SL_{2}(\mathbb{R})/SL_{2}(\mathbb{Z})$, it is morally equivalent to knowing the continued fraction expansion of the point of tangency defining the horosphere, see~\cite[Lemma~$2.10$]{sarnak-ubis}).

It seems reasonable to consider the diophantine condition defined by Strombergsson for horospherical actions, as the diagonal action is prominent. Moreover, given a $U$-invariant measure having such a diophantine point as a generic point, easy application of the linearization technique show that the given measure must be the Haar measure (indeed, for any fixed compact subset of a ``tube'', applying the $a$-flow to this piece, we can see that the $U$-orbit through that base point raise slower than the compact piece of the tube, hence after long enough time, the separation would force the orbit to spend only a negligible amount of time near this piece of tube).

In this subsection we generalize Strombergsson's result to general spaces of the form $X=G/\Gamma$, for the case of $G=SL_{n}(\mathbb{R}), \Gamma=SL_{n}(\mathbb{Z})$.
In particular, we show that (non-periodic) \emph{algebraic points} satisfy such a diophantine condition (in an \emph{effective manner}) and hence satisfy a quantitative equidistribution statement.

\subsubsection{Review of horospherical equidistribution}
Below we summarize the proof of quantitative equidistribution of horospherical orbits given in~\cite{katz}.
We fix for once and for all a split Cartan element $a\in G$, and define the associated horospherical subgroup $$H=\left\{g\in G \mid a_{t}ga_{-t} \to e \text{ as }t\to -\infty \right\}.$$
We fix the unit ball in $Lie(H)$ and we call its image under the exponential map $B_{1}^{H}$. We define for general $R$ the set $B_{R}^{H}$ as 
\begin{equation*}
    B_{R}^{H}=a_{\log R}B_{1}^{H}a_{-\log R}.
\end{equation*}
The sets $\left\{ B_{R}^{H}\right\}_{R\in \mathbb{R}}$ form an increasing Folner sequence.
For a smooth function of compact support $f\in C^{\infty}_{c}(G/\Gamma)$, we define the averaging operator $A_{R}f(x)$ to be
\begin{equation}
    A_{R}f(x)=\frac{1}{\vol_{H}(B_{R}^{H})}\int_{h\in B^{R}_{H}}f(h.x)dh,
\end{equation}
where $\vol_{H}$ is defined by the (unimodular) Haar measure on $H$.

Moreover, we have the following alternative description of the averaging operator (due to the renormalization of the $a$-action)
\begin{equation*}
     A_{R}f(x)=\frac{1}{\vol_{H}(B_{1}^{H})}\int_{h\in B^{1}_{H}}f(a_{\log R}.h.a_{-\log R}.x)dh.
\end{equation*}

\begin{prop}
    There exists some $s=s(\Gamma)>0$ such that 
    \begin{equation*}
        m\left\{x\in G/\Gamma \mid \left\lvert A_{R}f(x) - \int_{G/\Gamma}fdm \right\rvert>\varepsilon \right\} \leq \varepsilon^{-2}\cdot R^{-2s}.
    \end{equation*}
\end{prop}
The proof is based by observing that the function $q(x)=f(x)-~\int_{G/\Gamma}fdm$ belongs to $L^{2}_{0}(G/\Gamma)$ and the representation of $G$ on $L^{2}_{0}(G/\Gamma)$ posses a spectral gap, hence $\lVert A_{R}f \rVert_{L^{2}(m)}^{2} \leq R^{-s}$ and combining this estimate with Markov's inequality we deduce the proposition.
Picking $\varepsilon=R^{-\gamma}$ for $\gamma<s$ gives an effective mean ergodic theorem.

\begin{prop}
    Assume $x,y\in G/\Gamma$ such that $dist(a_{-\log R}.x,a_{-\log R}.y)<~\delta$ then $\lvert A_{R}f(x) - A_{R}f(y) \rvert \ll_{f} \delta$.
\end{prop}
This proposition follows from the horospherical action and our explicit choice of Folner sequence.

Combining those two propositions together, we get the following quantitative estimate.
\begin{prop}
    Assume that the following mass condition holds
    \begin{equation}\label{eq:mass-cond}
        m(B_{\delta}(a_{-\log R}.x)) > R^{2\gamma-2s},
    \end{equation} then 
    \begin{equation}\label{eq:semi-eff}
        \lvert A_{R}f(x) - \int_{G/\Gamma}fdm \rvert \ll_{f} \delta + R^{-\gamma}.
    \end{equation}
\end{prop}

In order to conclude the proof as in~\cite{katz}, one needs to verify the mass condition~\eqref{eq:mass-cond} and later set $\delta=R^{-\eta}$ and optimize the resulting~\eqref{eq:semi-eff}.

The main caveat is that the mass condition~\eqref{eq:mass-cond} \emph{does not need to hold}, indeed even in the case of $X=SL_{2}(\mathbb{R})/SL_{2}(\mathbb{Z})$ one may exhibit $H$-generic points $x$ (namely $\overline{H.x}=X$), but the ``height'' of the point $a_{-\log R}.x$ (which here can be considered as the height of the representative in the upper half plane model with respect to the usual Dirichlet fundamental domain) is high , even as high as $O(R^{1/2-\epsilon'})$ for some $\epsilon'$, as a result, the injectivity radius around $a_{-\log R}.x$ is of magnitude $O(R^{-1/2+\epsilon'})$ and so in the mass condition~\eqref{eq:mass-cond} we get
\begin{equation*}
    R^{-1+2\epsilon'}>R^{2\gamma-2s}.
\end{equation*}
As $s\leq 1/2$ and $\epsilon'$ is arbitrarily small, one cannot guarantee the existence of such conditions.
Such points are of ``Liouvillian'' nature as their dynamics is well-approximated by an ensemble of $H$-periods with relatively small period (compared to $R$).

In the previous approaches toward equidistribution, one used the uniform non-divergence theorem of Dani and Margulis, combined with some polynomial recurrence condition in order to force the $H$-orbit of the point $a_{-\log R}.x$ to a compact set where one may control the injectivity radius.
As unipotent trajectories diverge polynomially, once a trajectory intersects a compact set, it stays near this compact set for most of the time, as can be deduced (in a quantitative form) from Kleinbock-Margulis~\cite{kleinbock-margulis}.

Here we will use a different approach.
We may write the unipotent average as
\begin{equation}\label{eq:uni-avg-alt}
   B_{R}(f)(x) = \int_{B_{1}^{H}}f(a_{\log R}.h.a_{-\log R}.x)dh.
\end{equation}
We will say that $x$ is diophantine if the ``height'' of $a_{-\log R}.x$ grows slowly, in particular we want the height to grow in a speed which is lesser than the speed defined by all the various tubes containing $H$-periods, see Definition~\ref{def:dioph} for precise statement.
We are then going to use the proof of Lemma~\ref{lem:main-lem} in order to show that there exists some $\tau=~\tau(\varepsilon,\alpha)<1$ for which
\begin{equation*}
    a_{\tau\log R}.B_{1}^{H}.a_{-\log R}.x \cap X_{\geq\varepsilon} \neq\emptyset.
\end{equation*}
Using this result, we may modify the average~\eqref{eq:uni-avg-alt} as follows
\begin{equation}
\begin{split}
    B_{R}(f)(x) &= \int_{B_{1}^{H}}f(a_{\log R}.h.a_{-\log R}.x)dh \\
    &= \int_{B_{1}^{H}}f(a_{(1-\tau)\log R}.a_{\tau\log R}.h.a_{-\log R}.x)dh \\
    &= \int_{h\in B_{1}^{H}} \int_{u\in B_{1}^{H}}f(a_{(1-\tau)\log R}.u.(a_{\tau\log R}.h.a_{-\log R}.x))dudh \\
    &\ \  + O_{f,\tau,H}\left(R^{-\tau\cdot d_{H}}\right).
\end{split}
\end{equation}
where $d_{H}$ is some constant depends on $H$, the error was introduced because of the extra averaging (convolution) by $u$ in the appropriate scale.

Now by the result we have about the divergence, for most of the $h\in B_{1}^{H}$, $a_{\tau\log R}.h.a_{-\log R}.x \in X_{\geq \varepsilon}$.
For such $h$'s, we have that the inner integral
$\int_{u\in B_{1}^{H}}f(a_{(1-\tau)\log R}.u.(a_{\tau\log R}.h.a_{-\log R}.x))du$ converges to $\int fdm$ quantitatively and we may bound the rest of the integrand trivially by $O_{f}(m(\left\{h\in B_{1}^{H} \mid a_{\tau\log R}.h.a_{-\log R}.x \notin X_{\geq \varepsilon}  \right\}))$.

The estimate for the measure of the bad set follows from the Margulis' inequality by using Markov's inequality as follows:
\begin{equation}
\begin{split}
 m\left\{\underline{h}\leq 1 \mid \alpha_{i}(a_{\log r}.\exp(\underline{h}).a_{-\log R}y)< \varepsilon^{1/\delta}  \right\} &= m\left\{\underline{h}\leq 1 \mid \alpha_{i}(a_{\log r}.\exp(\underline{h}).y)^{-\delta}>\varepsilon  \right\} \\
 &\leq \frac{A_{r}\alpha_{i}^{-\delta}(y)}{\varepsilon} \\
 &\leq \frac{C_{r,\delta}\cdot\alpha_{i}^{-\delta}(y)}{\varepsilon},
 \end{split}
 \end{equation}
 writing $y=a_{-\log R}.x$ and using the diophantine assumptions over the divergence of $x$ under the $a_{-\log R}$-flow we deduce the required estimate.

\subsection{Diophantine properties of lattices defined over number fields}
The corollary for equidistribution of algebraic lattices will follow at once, once we prove that for algebraic lattices $x$, the diophantine exponents $\alpha_{i}(x)$ is (effectively) strictly less than $1$ for any $i$.

We begin this by showing the following example for the $SL_{2}(\mathbb{R})$ case.
\begin{ex}
Consider $a_{-\log T}$ as needed.
For any $x$ which is in $P$ - the upper triangular subgroup (which is a $\mathbb{Q}$-parabolic subgroup containing $H$) we have the following ``rate'' of escape to the cusp - by considering $v=(1,0)$ we have that
\begin{equation}
    \limsup_{T\to \infty}\frac{\log \lVert \rho(a_{-\log T}).v\rVert}{0.5\log T} = 1.
\end{equation}
Now assume that $x$ is an algebraic matrix in $SL_{2}(\mathbb{R})$ which is not contained in $P$ then we want to show that the shortest non-zero vector in $a_{-\log T}.x.\mathbb{Z}^{2}$ is longer than $O(T^{-1/2})$.

Pick some $M,N\in \mathbb{Z}$ such that $a_{-\log T}.x.(M,N)$ is the shortest non-zero vector.
We have that
\begin{equation*}
    a_{-\log T}.x.(M,N)=(T^{-1/2}\cdot(aM+bN),T^{1/2}\cdot(cM+dN)).
\end{equation*}
As $x\notin P$ we have that $c\neq 0$.
If $\lvert cM+dN \rvert \gg T^{-1+\delta}$ for some small $\delta>~0$, then we are done.
In the other case we have that $T^{-1+\delta}\gg \lvert cM+dN \rvert$ and moreover $\lvert cM+dN \rvert \gg M^{-1-\epsilon}$ for any $\epsilon$ where this inequality is non-effective due to \emph{Roth's theorem}.
So we have that $M\gg T^{(1-\delta)/(1+\epsilon)}$.
Furthermore, as $\lvert cM+dN \rvert \approx 0$, we have that $\lvert aM+bN \rvert \gg M$ so $\lvert T^{-1/2}\cdot (aM+bN) \rvert \gg T^{-1/2}\cdot M \gg T^{-1/2+(1-\delta)/(1+\epsilon)}$.
So one may show (ineffectively) that the diophantine rate of $x$ is less than $\delta$ for any positive $\delta$.

Now we may make the above argument effective in the following manner.
Using Liouville's theorem, we may replace the ineffective bound derived from Roth's theorem in an effective bound (depending on the field of definition of $x$), namely
$\lvert cM+dN \rvert \gg M^{-(d-1)}$ where $[\mathbb{Q}(g):\mathbb{Q}]\leq d$, where this inequality is \emph{effective}.
As a result, we have that
$M\gg T^{(1-\delta)/(d-1)}$.
If so we get that $\lvert T^{-1/2}(aM+bN)\rvert \gg T^{-1/2}\cdot T^{(1-\delta)/(d-1)}$.
Optimizing both regions lead to considering $T^{1/2-(1-\delta)}$ and $T^{-1/2+(1-\delta)/(d-1)}$, which gives $\delta=\frac{1}{d}$, hence the diophantine exponents is estimated effectively as
\begin{equation*}
    \alpha_{1}(x) \leq \frac{-1/2+\delta}{-1/2} = 1-\frac{\delta}{2}.
\end{equation*}
\end{ex}

Now we are going to generalize this notion of being diophantine from $SL_{2}$ to general algebraic groups defined over $\mathbb{Q}$. Using this notion, one may show that lattices defined over algebraic number fields (which are not $H$-invariant) will satisfy a diophantine condition effectively. Using the method discussed above, one deduce a fully quantified and effective horospherical equidistribution theorem for such class of lattices.

Consider the settings of Dani-Margulis~\cite{Dani1991}. Let $G$ be the real points of a semisimple algebraic group defined over $\mathbb{Q}$.
Let $r$ be the $\mathbb{Q}$-rank of $G$, which we assume to be at-least $1$.
Let $S$ be the real points of a maximal split $\mathbb{Q}$-torus in $G$.
We fix an ordering over th $\mathbb{Q}$-roots of $S$, and denote by $\left\{\alpha_{1},\ldots,\alpha_{r} \right\}$ the set of simple $\mathbb{Q}$-roots.
For $i=1,\ldots,r$, let $P_{i}$ be the standard maximal $\mathbb{Q}$-parabolic subgroup corresponding to the set of roots different from $\alpha_{i}$.
For each $i$, the root $\alpha_{i}$ which is a character over $S$, extends uniquely to a character over $P_{i}$, which we will denote also by $\alpha_{i}$.
For $1\leq i \leq r$ denote by $U_{i}=Rad^{u}(P_{i})$ be the unipotent radical of $P_{i}$.
For each $x\in P_{i}$, we get that $Ad(x)$ acts as a linear operator over $Lie(U_{i})$, with $\det (Ad(x)\mid_{Lie(U_{i})}) = \alpha_{i}^{m_{i}}(x)$ where $m_{i}=\sum n_{\lambda}\cdot\lambda_{i}$ where the sum is taken over all positive roots $\lambda$ where for each such $\lambda$ $n_{\lambda}$ is the dimension of the root space corresponding to $\lambda$ and $\lambda_{i}$  equals to the coefficients of $\alpha_{i}$ in the expansion of $\lambda$ in terms of $\alpha_{1},\ldots,\alpha_{r}$.

We fix a maximal compact subgroup $K$ of $G$ for which $S$ is invariant under the Cartan involution of $G$ associated to $K$, and we have the Cartan decomposition of $G$ as follows $G=K\cdot P_{i}$. 
We define the functions $d_{i}(g)=~\lvert\alpha_{i}(x)\rvert^{m_i}$, where $g=k\cdot x$, $k\in K,x\in P_{i}$, for $i=1,\ldots,r$.
We have the following definition of diophantine exponents, for $i=1,\ldots,r$
\begin{defn}\label{def:dioph}
\begin{equation}\label{eq:def-dioph}
    \alpha_{i}(g) = \limsup_{T\to \infty} \frac{\log( \min_{\gamma\in \Gamma,f\in F} d_{i}(a_{-\log T}.g.\gamma.f))}{\log(d_{i}(a_{-\log T}))}.
\end{equation}
\end{defn}
This definition agrees with the definition given in the introduction.

We say that $g$ satisfy a diophantine condition if
$\alpha_{i}(g)<1$ \emph{effectively}, for all $i=1,\ldots,r$.
\begin{thm}
Assume $g$ is defined over some number field, and $g\notin~P$ for any $\mathbb{Q}$-parabolic subgroup $P$, then $g$ is diophantine.
\end{thm}

\begin{proof}
Consider $\bigwedge^{i} Lie(G)$.
We fix a primitive basis for $Lie(H)$ and for that we associate with $H$ a vector $v_{H}\in \bigwedge^{i} Lie(G)$ which is fixed under the $H$, and moreover, $a_{-\log T}.v_{H}=d_{i}(a_{-\log T}).v_{H}$.
Now consider the vector $x.\gamma.f.v_{H}$, we may write
$x.\gamma.f.v_{H}=\rho_{v_{H}}(x.\gamma.f.v_{H})\cdot \hat{v_{H}} + \sum_{j}\rho_{j}(x.\gamma.f.v_{H}).v_{j}$, where $v_{j}$ are formed from a rational basis to $\bigwedge^{i} Lie(G)$.
As $x$ is not contained in any proper parabolic $P$, we must have that some of the projections $\rho_{j}(x.\gamma.f.v_{H})$ to the other $v_{j}$ which are not contracted under the $a_{-\log T}$-action (by definition of $H$ as being the contracted set) are non-zero.

Now assume that for all of those non-contracting directions $v_{j}$ we have that $\rvert a_{-\log T}.\rho_{j}(x.\gamma.f.v_{H})\rvert \leq d_{i}(a_{-\log T})^{1-\epsilon}.$
Write $$a_{-\log T}.v_{j}=d_{i}(a_{-\log T})^{\beta_{j}},$$ for some $b_{j}< 1$.
We must have that
$$ \lvert \rho_{j}(x.\gamma.f.v_{H})\rvert \leq d_{i}(a_{-\log T})^{1-\epsilon-\beta_{j}}. $$
Note that as $x$ is algebraic, and we have that for a fixed $x,f,v_{H}$, $\rho_{j}(x.\gamma.f.v_{H})$ is a linear form defined over some number field, we must have that
$$ \rho_{j}(x.\gamma.f.v_{H}) \gg_{x} \text{height}(f)^{-D(x)} ,$$
by Liouville's theorem, where we define the height of an integral vector as its height inherited from the arithmetic structure (as $G$ is embedded beforehand algebraically in some $GL_{N}$ it inherits the height function from $GL_{N}$).

Hence $\text{height}(\gamma)\gg_{x} d_{i}(a_{-\log T})^{-(1-\epsilon-\beta_{j})/D}$.
As $\rho_{v_{H}}(x.\gamma.f.v_{H})$ is a linear form, we must have that $$\rho_{v_{H}}(x.\gamma.f.v_{H}) \gg_{x} \text{height}(\gamma) \gg_{x} d_{i}(a_{-\log T})^{-(1-\epsilon-\beta_{j})/D}. $$
As a result
$$\lvert a_{-\log T}.\rho_{v_{H}}(x.\gamma.f.v_{H}).v_{H}\rvert \gg _{x} d_{i}(a_{-\log T})^{1-(1-\epsilon-\beta_{j})/D}.$$

As we have that $\beta_{j}<1$, for $\epsilon>0$ small enough, we get a non-trivial diophantine condition.
\end{proof}

\begin{rem}
Using the subspace theorem, one may show that \emph{non-effectively}, one may replace $D(x)$ with a number arbitrarily close to $1$, leading to show a sublinear divergence rate. Hence one may conclude that \emph{non-effectively}, in the case of lattices defined over number fields, the equidistribution rate under horospherical flows is the same as the rate achieved in~\cite{katz} for compact quotients.
\end{rem}

\bibliographystyle{plain}
\bibliography{margulis-ineq}

\end{document}